\documentclass[10pt]{amsart}

\usepackage{graphicx}

\newtheorem{theorem}{Theorem}

\newtheorem{proposition}{Proposition}
\newtheorem{lemma}{Lemma}

\title{A continuation of solutions to convolution equations with the loss of smoothness}
\author{Anastasiia Minenkova}
\address{Department of Mathematics, The University of Mississippi, Oxford, MS, USA}

\email{aminenko@olemiss.edu}
\begin{document}
\begin{abstract}
In the present paper the smoothness loss of a continuation of solutions to
convolution equations is studied. Also examples for some kinds of
convolvers are given.
\end{abstract}
\maketitle



\section{Introduction}

The theory of mean periodic functions is a subject that goes back
to works of J. Littlewood, J. Delsarte, F. John. As a matter of fact, mean periodic functions were introduced by J. Delsarte~\cite {D} and then refined by L. Schwartz~\cite {S}. After that, the theory has undergone vigorous development and, in particular, there has been much progress
on the matter that influenced the development of some problems concerning local aspects of spectral analysis
and spectral synthesis on homogeneous spaces. In fact, one of such problems has motivated the research presented in this note. More precisely, our interest comes from the fact that mean periodic functions appear as solutions to homogeneous convolution equations, which turn out to be an efficient tool in many areas of mathematics and applications. Moreover, as one can see from
papers by M. Agranovsky, P. Kuchment, L. Zalcman, V.V. Volchkov, Vit.V. Volchkov and many others, the
study of these problems is closely related to a variety of questions in harmonic analysis, complex analysis, partial differential equations, integral
geometry, approximation theory, probability theory, stochastic processes, etc. (see~\cite {AK}--\cite {V4}). A
good introduction to the subject of mean periodic functions and convolution equations is a book by V.V. Volchkov and
Vit.V. Volchkov~\cite {V3}.

In this paper we are going to focus on the problem of continuation of solutions to homogeneous convolution equations. Also, in Section 4 we'll discuss explicit examples for
some kinds of convolvers.

This work continues the study of the properties of exponential polynomials and the behavior of extensions of solutions to convolution equations (see ~\cite{L}--\cite {M2}).

\section{Notations and Auxiliary Statements}

Let  $\mathcal{D}'(a, b)$ be a class of all distributions. Let $T\in\mathcal{E}'(\mathbb{R}^1)$, $ T~\neq~0$, where $\mathcal{E}'(\mathbb{R}^1)$ is the space of compact supported distributions. Let \[\mathrm{supp}\,T~=~[-r(T),r(T)].\]

Suppose that
\[
-\infty\leq a<b\leq +\infty, \,b-a>2r(T).
\]

Let us introduce the following notation:
\[
(a,b)_T=\{t\in\mathbb{R}^1: t-\mathrm{supp}\,T\subset(a,b)\}.
\]

We denote by $\mathcal{D}_T' (a, b)$ the class of all distributions  $f \in \mathcal{D}'(a, b)$ that are solutions of the convolution equation
\begin{equation*}\label{eq1}
(f*T)(t)=0, \quad t\in(a,b)_T.
\end{equation*}
Such $f$ are also called mean periodic.

Moreover,
\[
C_T^k(a,b)=(\mathcal{D}_T'\cap C^k)(a,b) \text{ for }k\in \mathbb{Z}_+ \text{ or }k=\infty.
\]

Let $\widehat{T}=\langle T,e^{-izt}\rangle$ be the Fourier transform of $T$ and let $\mathcal{Z}(\widehat{T})$ be
the set of zeroes of  $\widehat{T}$. For $\lambda\in\mathcal{Z}(\widehat{T})$ denote 
$m(\lambda,T)=n_\lambda(\widehat{T})-1$, where $n_\lambda(\widehat{T})$ ($n_\lambda=n_\lambda(\widehat{T})$) is the multiplicity of the zero $\lambda$ of $\widehat{T}$.

 There is a well-known fact from the theory of entire functions that for each
$\varepsilon>0$ one has
\begin{equation*}\label{eq1_1}
{\underset{\lambda\in\mathcal{Z}(\widehat{T})}{\sum}}\frac{n_\lambda}{|\lambda|^{1+\varepsilon}}<+\infty.
\end{equation*}

Let us define the sequence
$\{a_j^{\lambda,\eta}(\widehat{T})\},\,j=0,...,m(\lambda,T)$ as follows
\[
a_0^{\lambda,\eta}(\widehat{T})=\frac{n_\lambda!\delta_{0,\eta}}{\widehat{T}^{(n_\lambda)}(\lambda)},
\]
\[
a_j^{\lambda,\eta}(\widehat{T})=\frac{n_\lambda!}{\widehat{T}^{(n_\lambda)}(\lambda)}\left(\frac{\delta_{j,\eta}}{j!}-
\overset{j-1}{\underset{s=0}{\sum}}
a_s^{\lambda,\eta}(\widehat{T})\frac{\widehat{T}^{(n_\lambda-s+j)}(\lambda)}{(n_\lambda-s+j)!}
\right),\; j\geq 1,
\]
where $\delta_{j,\eta}$ is the Kronecker symbol.

Introduce the function
\[
\sigma_\lambda(\widehat{T})=\overset{m(\lambda,T)}{\underset{j=0}{\sum}}
|a_j^{\lambda,0}(\widehat{T})|.
\]

In what follows, we need the following entire function
\[
a^{\lambda,\eta}(\widehat{T},z)=\overset{m(\lambda,T)}{\underset{j=0}{\sum}}a_j^{\lambda,\eta}(\widehat{T})\frac{\widehat{T}(z)}{(z-\lambda)^{n_\lambda-j}}.
\]

For $z\in\mathbb{C},m\in\mathbb{Z}_+, t\in\mathbb{R}^1$ we denote
\[
e^{z,m}(t)=(it)^m e^{izt}.
\]

Let $T\in\mathcal{E}'(\mathbb{R}^1)$, $ T\neq 0$,
$\lambda\in\mathcal{Z}(\widehat{T})$, $\eta\in
\{0,...,m(\lambda,T)\}$ and $f\in \mathcal{D}_T'(a,b)$. One can show that for some $c_{\lambda,\eta}(T,f)\in\mathbb{C}$
the following equality holds
\[
f\ast T_{\lambda,0} =
\overset{m(\lambda,T)}{\underset{\eta=0}{\sum}}c_{\lambda,\eta}(T,f)e^{\lambda,\eta},
\]
where the convolution is considered in $(a,b)_T$ and
$T_{\lambda,0}\in\mathcal{E}'(\mathbb{R}^1)$ is defined in the following way
\[
r(T_{\lambda,0})=r(T),
\]
and
\[
\widehat{T}_{\lambda,0}(z) =
a^{\lambda,0}(\widehat{T},z),\,z\in\mathbb{C}.
\]

\begin{proposition}\label{st1}$($\cite[Theorem 3.9(ii)]{V3}$)$ Let $T\in\mathcal{E}'(\mathbb{R}^1),\, T\neq 0$ and $f\in
\mathcal{D}'(a,b)$. Suppose that
\[
f={\underset{\lambda\in\mathcal{Z}(\widehat{T})}{\sum}}\overset{m(\lambda,T)}{\underset{\eta=0}{\sum}}c_{\lambda,\eta}e^{\lambda,\eta},
\]
where $c_{\lambda,\eta}\in\mathbb{C}$ and that the series converges in 
$\mathcal{D}'(a,b)$. Then $f\in \mathcal{D}_T'(a,b)$ and
$c_{\lambda,\eta}=c_{\lambda,\eta}(T,f)$.
\end{proposition}

Let us set for $\lambda\in\mathcal{Z}(\widehat{T})$, $R>0$ and
$q\in\mathbb{Z}_+$:
\begin{equation}\label{eqB}
B(R,\lambda,q)=
  \begin{cases}
    R^{m(\lambda,T)} & \text{if }{R>1}, \\
    m(\lambda,T)+1  & \text{if }{R=1}, \\
    \min\{q+1,m(\lambda,T)+1\} & \text{if }{R<1}.
  \end{cases}
\end{equation}

The following results, that were obtained in~\cite {V3},  will be useful for us in order to prove the main result.

\begin{proposition}\label{st2}$($\cite[Proposition 2.27(ii)]{V3}$)$
Let $r>0$ and let
\begin{equation}\label{eq2}
\underset{\lambda\in\mathcal{Z}(T)}{\sum}
\left(\underset{0\leq\eta\leq
m(\lambda,T)}{\max}|c_{\lambda,\eta}|\right)
B(r,\lambda,k)(|\lambda|+1)^k e^{r|\mathrm{Im}\lambda|}<+\infty,
\end{equation}
for some $k\in \mathbb{Z}_+$. Then the series
\begin{equation}\label{eq12}
\underset{\lambda\in\mathcal{Z}(T)}{\sum}\overset{m(\lambda,T)}{\underset{\eta=0}{\sum}}
c_{\lambda,\eta}e^{\lambda,\eta}
\end{equation}
converges in $C^k[-r,r]$.
\end{proposition}

\begin{proposition}\label{st3} $($\cite[Theorem 3.10]{V3}$)$
Let $T\in\mathcal{E}'(\mathbb{R}^1)$, $T\neq 0$. Then the following hold true.
\begin{itemize}
  \item [(i)] Let $f\in\mathcal{D}'_T(a,b)$and let $p$ be a nonzero polynomial. Then there exist $\gamma_1,\gamma_2>0$ independent of $f$ such that for all $\lambda\in \mathcal{Z} (\widehat{T})$,
  $|\lambda|>\gamma_1$ the following estimate holds
\begin{equation*}\label{eqst1}
\underset{0\leq\eta\leq
m(\lambda,T)}{\max}|c_{\lambda,\eta}(T,f)|\leq
\frac{\gamma_2}{|p(i\lambda)|}\underset{0\leq\eta\leq
m(\lambda,T)}{\max}\left|c_{\lambda,\eta}\left(T,p\left(\frac{d}{dt}\right)f\right)\right|.
\end{equation*}
  \item [(ii)]
 Let $k\in\mathbb{Z}_+$, $f\in C_T^k(-R,R)$ and $R>r(T)$. Then there exist  $\gamma_3$,$\gamma_4$, $\gamma_5$, $\gamma_6>0$
independent of $\,k$ and $f$ such that for all $\lambda\in\mathcal{Z}(\widehat{T}):\,|\lambda|>\gamma_3$ 
\begin{equation}\label{eq3}
\underset{0\leq\eta\leq
m(\lambda,T)}{\max}|c_{\lambda,\eta}|\leq\gamma_4^{k+1}|\lambda|^{\gamma_5-k}\sigma_\lambda(\widehat{T})
\left(\int_{-r(T)}^{r(T)}
|f^{(k)}(t)|dt+\gamma_6^k\gamma_7\right),
\end{equation}
where $\gamma_7>0$ is independent of $k$, $\lambda$.
\end{itemize}
\end{proposition}

Next, we are going to prove the technical lemma.
\begin{lemma}\label{l21}
Let $T\in\mathcal{E}'(\mathbb{R}^1)$, $T\neq 0$ and let \eqref{eq2} be true for some
$k\in\mathbb{N}$ and $r>0$. If for $R>r$
there exists $N\in\mathbb{N}:\;|\lambda|>N$ such that
\begin{equation}\label{eql21}
\underset{|\lambda|>N}{\sup}
\frac{|\mathrm{Im}\lambda|+m(\lambda,T)}{\ln(2+|\lambda|)}<\frac{1}{R-r},
\end{equation}
then \eqref{eq12} converges in $C^{k-1}[-R,R]$.

\end{lemma}

\begin{proof} 
Let $|\lambda|>N$, where $N$ is such that \eqref{eql21} is true. Consider the following casesapplied to \eqref{eqB}.
\begin{itemize}
  \item [1)] Let $R>r>1$.  If $q\leq
  k-\frac{(R-r)|\mathrm{Im}\lambda|+\ln\frac{R}{r}m(\lambda,T)}{\ln(1+|\lambda|)}$
  then we get
  \begin{equation}\label{eqlp1}
  \frac{B(R,\lambda,q)(|\lambda|+1)^q
e^{R|\mathrm{Im}\lambda|}}{B(r,\lambda,k)(|\lambda|+1)^k
e^{r|\mathrm{Im}\lambda|}}=\frac{R^{m(\lambda,T)}(|\lambda|+1)^q
e^{R|\mathrm{Im}\lambda|}}{r^{m(\lambda,T)}(|\lambda|+1)^k
e^{r|\mathrm{Im}\lambda|}}\leq 1.
\end{equation}
  \item [2)] Let $R>r=1$. If $q\leq
  k-\frac{(R-1)|\mathrm{Im}\lambda|+m(\lambda,T)\ln{R}-\ln(m(\lambda,T)+1)}{\ln(1+|\lambda|)}$
  then one has
  \begin{equation}
  \frac{B(R,\lambda,q)(|\lambda|+1)^q
e^{R|\mathrm{Im}\lambda|}}{B(r,\lambda,k)(|\lambda|+1)^k
e^{r|\mathrm{Im}\lambda|}}=\frac{R^{m(\lambda,T)}(|\lambda|+1)^q
e^{R|\mathrm{Im}\lambda|}}{{(m(\lambda,T)+1)}(|\lambda|+1)^k
e^{|\mathrm{Im}\lambda|}}\leq 1.
\end{equation}
  \item [3)] Let $R=1>r$.  If $q\leq
  k-\frac{(1-r)|\mathrm{Im}\lambda|-m(\lambda,T)\ln{r}}{\ln(1+|\lambda|)}$
  then we arrive at
  \begin{equation}
  \frac{B(R,\lambda,q)(|\lambda|+1)^q
e^{R|\mathrm{Im}\lambda|}}{B(r,\lambda,k)(|\lambda|+1)^k
e^{r|\mathrm{Im}\lambda|}}=\frac{{m(\lambda,T)}(|\lambda|+1)^q
e^{|\mathrm{Im}\lambda|}}{\min\{k+1,{m(\lambda,T)}\}(|\lambda|+1)^k
e^{r|\mathrm{Im}\lambda|}}\leq 1.
\end{equation}
  \item [4)] Let $1>R>r$.  If $q\leq
  k-\frac{(R-r)|\mathrm{Im}\lambda|+\ln\frac{R}{r}m(\lambda,T)}{\ln(1+|\lambda|)}$
  then one has
  \begin{equation}\label{eqlp4}
  \frac{B(R,\lambda,q)(|\lambda|+1)^q
e^{R|\mathrm{Im}\lambda|}}{B(r,\lambda,k)(|\lambda|+1)^k
e^{r|\mathrm{Im}\lambda|}}=\frac{(|\lambda|+1)^q
e^{R|\mathrm{Im}\lambda|}}{(|\lambda|+1)^k
e^{r|\mathrm{Im}\lambda|}}\leq 1.
\end{equation}
\end{itemize}

By looking at \eqref{eql21} we can say that if $q=k-1$ then
\eqref{eqlp1}--\eqref{eqlp4} hold true. Hence, the following transformations are true for arbitrary $R$ and $r$ ($R>r$)
\begin{equation}\label{eql}
\underset{\lambda\in\mathcal{Z}(T)}{\sum}
\left(\underset{0\leq\eta\leq
m(\lambda,T)}{\max}|c_{\lambda,\eta}|\right)
B(R,\lambda,q)(|\lambda|+1)^q e^{R|\mathrm{Im}\lambda|}
\end{equation}
\[=\underset{\lambda\in\mathcal{Z}(T)}{\sum}
B(r,\lambda,k)(|\lambda|+1)^k e^{r|\mathrm{Im}\lambda|}
\frac{B(R,\lambda,q)(|\lambda|+1)^q
e^{R|\mathrm{Im}\lambda|}}{B(r,\lambda,k)(|\lambda|+1)^k
e^{r|\mathrm{Im}\lambda|}}\underset{0\leq\eta\leq
m(\lambda,T)}{\max}|c_{\lambda,\eta}|
\]
\[
\leq\underset{\lambda\in\mathcal{Z}(T), |\lambda|>N}{\sum}
\left(\underset{0\leq\eta\leq
m(\lambda,T)}{\max}|c_{\lambda,\eta}|\right)
B(r,\lambda,k)(|\lambda|+1)^k e^{r|\mathrm{Im}\lambda|}
\]
\[
+\underset{\lambda\in\mathcal{Z}(T),|\lambda|\leq N}{\sum}
\left(\underset{0\leq\eta\leq
m(\lambda,T)}{\max}|c_{\lambda,\eta}|\right)
B(R,\lambda,q)(|\lambda|+1)^q e^{R|\mathrm{Im}\lambda|}<\infty.
\]

Now we see that the statement of lemma is implied by Proposition~\ref{st2} and \eqref{eql}.
\end{proof}

\section{Formulation and Proof of the Main Result.}

\begin{theorem}\label{t31}
Assume that $T\in\mathcal{E}'(\mathbb{R}^1)$, $T\neq 0$, $f\in
C^k_T(-r,r)$, where $r>r(T)>0$. Also, assume there exists $q\in\mathbb{Z}_+$ such that
\begin{equation}\label{eqt31}
\underset{\lambda\in\mathcal{Z}(T)}{\sum}
\sigma_\lambda(\widehat{T})
B(r,\lambda,q+1)(|\lambda|+1)^{\gamma-k+q+1}
e^{r|\mathrm{Im}\lambda|}<+\infty,
\end{equation}
where $\gamma>0$ is independent of $f$, $k$, $\lambda$. Then if there exists
$N>0$ such that for $R>r$  the estimate \eqref{eql21} holds then $f\in
C^q_T(-R,R)$, where $q<k-2-\gamma$.
\end{theorem}

\begin{proof}
The condition $ f\in C^k_T (-r, r)$ means that the function $ f $ can be represented as the series~\eqref{eq12} (see Proposition~\ref{st1}) and for $c_{\lambda,\eta} $ the estimate \eqref{eq3} follows from Proposition~\ref{st2}.
 Using~\eqref{eq3} and \eqref{eqt31} we obtain \eqref{eq2} for $ q+1 $ and $ r $. 
This means that the conditions of Lemma~\ref{l21} are satisfied, i.e. $f\in C^q_T(-R,R)$.
\end{proof}

\section{Examples of an Extension of Solutions for Some Convolvers.}

To illustrate the loss of smoothness in some cases we give
a couple of explicit examples.

 {\bf Example 1.} Consider the following form of the convolver
\begin{equation}\label{eq41}
T=
  \begin{cases}
    (r^2-t^2)^{\alpha-1/2} & \text{if }{|t|\leq r}, \\
    0 & \text{if } {|t|>r},
  \end{cases}
\end{equation}
where $\alpha>-1/2$. Then
\[
\widehat{T}(t)=\frac{c \mathcal{J}_\alpha(rt)}{t^\alpha},
\]
where $\mathcal{J}_\alpha$ is the~Bessel function and $c=c(r,\alpha)>0$.

The asymptotic behavior of zeros of the Bessel functions is known to be
$\zeta_m=\pi(m+\frac{2\alpha-1}{4})+O(\frac{1}{m})$ as
$m\rightarrow+\infty$ (see~\cite[p. 26]{V1}).
Also there is a well-known fact that
$|\mathcal{J}'_\alpha(\zeta_m)|\geq|\zeta_m|^{-1/2}$. Hence, it is easy to
see that
\begin{equation}\label{pr3}
\sigma_{\zeta_m}(\widehat{T})=\frac{c_1\zeta_m^\alpha}{\mathcal{J}'_\alpha(\zeta_m)}\leq
c_2 m^{\alpha+1/2}.
\end{equation}

\begin{theorem}\label{prop1}
If $\,f\in C_T^k[-R,R]$, where $k\in \mathbb{N}$, $T$ is defined
by formula \eqref{eq41} and $R>r$ then $f\in C_T^q(\mathbb{R})$ when
$\,q<k-(\alpha+\frac{3}{2})$.
\end{theorem}
\begin{proof}
According to the conditions of this theorem, $f\in C_T^k[-R,R]$, where $R>r$. Using
Proposition~\ref{st3}, we obtain
\begin{equation}\label{pr1}
|c_{\zeta_m}(T,f)|\leq
\frac{\gamma_1}{|\zeta_m|^k}|c_{\zeta_m}(T,f^{(k)})|=\frac{\gamma_1}{|\zeta_m|^k}\left|\int_{-r}^{r}T_{\zeta_m}(t)f^{(k)}(t)dt\right|
\end{equation}
\[
\leq
\frac{\gamma_1}{|\zeta_m|^k}\underset{[-r,r]}{\max}|T_{\zeta_m}(t)|
\int_{-r}^{r}|f^{(k)}(t)|dt \leq
\frac{\gamma_2}{|\zeta_m|^k}\sigma_{\zeta_m}(\widehat{T})
\int_{-r}^{r}|f^{(k)}(t)|dt,
\]
where $\gamma_1$, $\gamma_2>0$ is independent of $f$ and $\zeta_m$.

Consider now the condition \eqref{eq2} for $\widetilde{R}>R$. In this case it becomes 
\begin{equation}\label{pr2}
\underset{m=1}{\overset{\infty}{\sum}}
|c_{\zeta_m}|\max\{\widetilde{R},1\}(|\lambda|+1)^q<+\infty.
\end{equation}

Next, we need to find out under what conditions on $q$ and $k$, \eqref{pr2} holds.
It follows from \eqref{pr3} and \eqref{pr1} that
\[
\underset{m=1}{\overset{\infty}{\sum}}
|c_{\zeta_m}|\max\{\widetilde{R},1\}(|\lambda|+1)^q\leq
\underset{m=1}{\overset{\infty}{\sum}}
\frac{\gamma_3}{|\zeta_m|^k}\sigma_{\zeta_m}(\widehat{T})
(|\zeta_m|+1)^q 
\]
\[\leq \underset{m=1}{\overset{\infty}{\sum}}
\frac{\gamma_3}{m^{k-\alpha-1/2-q}},
\]
where $\gamma_3>0$ is independent of $\zeta_m$. That is the series \eqref{pr2} converges 
when $k-\alpha-1/2-q>1$ for every $\widetilde{R}$.  By using Proposition \ref{st2} we conclude that 
$f\in C^q_T(\mathbb{R}^1)$ for $q<k-(\alpha+3/2)$.
\end{proof}

{\bf Example 2.} Let us generalize a bit the previous example. So, set
\begin{equation}\label{eq42}
T=
  \begin{cases}
    (r^2-t^2)^{\alpha-1/2}h(t) & \text{if }{|t|\leq r}, \\
    0 & \text{if }{|t|>r},
  \end{cases}
\end{equation}
where $\alpha>-1/2$ and  $h(t)\in C^2[-r,r]$ is an even function. Then
\[
\widehat{T}(z)=\int_{-r}^r (r^2-t^2)^{\alpha-1/2}h(t)e^{izt} dt.
\]
\[
\widehat{T}'(z)=i\int_{-r}^r (r^2-t^2)^{\alpha-1/2}th(t)e^{izt} dt.
\]

The asymptotic behavior  of integrals of this kind at $z\rightarrow \infty$ is  well known (see ~\cite[p.~27]{V1}).
\begin{equation}\label{pr211}
c_{k,1} =\frac{(-1)^k\Gamma(k + \alpha)}{k!}\left(\frac{d}{dt}\right)^k((b - t)^{\beta - 1}h(t))|_{t=a},
\end{equation}
\begin{equation}\label{pr212}
c_{k,2} =\frac{(-1)^k\Gamma(k + \beta)}{k!}\left(\frac{d}{dt}\right)^k((t-a)^{\alpha - 1}h(t))|_{t=b},
\end{equation}
\begin{equation}\label{pr213}
\int_{a}^b (b-t)^{\beta-1}(t-a)^{\alpha-1}h(t)e^{izt} dt\sim e^{i(za+\alpha\pi)}\underset{k=0}{\overset{\infty}{\sum}}\frac{c_{k,1}}{(iz)^{\alpha+k}}+e^{izb}\underset{k=0}{\overset{\infty}{\sum}}\frac{c_{k,2}}{(iz)^{\beta+k}}.
\end{equation}

In the case of $\widehat{T}(z)$: $a=-r$, $b=r$, $\alpha=\beta=\alpha+1/2$. Note that as far as $h(t)$  is an even function, we have  $c_{0,1}=c_{0,2}$. Then  we rewrite the formula \eqref{pr213} in the following manner
\begin{equation}\label{pr214}
\widehat{T}(z) \sim \frac{(e^{i(-rz+(\alpha+1/2)\pi)}+e^{izr}){c_{0,1}}}{(iz)^{\alpha+1/2}}+ O\left(\frac{1}{|z|}\right).
\end{equation}

For $\,\widehat{T}'(z)$ in \eqref{pr211} and \eqref{pr212} instead of $h(t)$ one should substitute $th(t)$. In that case we get that $c_{0,1}'=-c_{0,2}'$, then formula \eqref{pr213} has the form
\begin{equation}\label{pr215}
\widehat{T}'(z) \sim \frac{i(e^{i(-rz+(\alpha+1/2)\pi)}-e^{izr}){c_{0,1}'}}{(iz)^{\alpha+1/2}}+ O\left(\frac{1}{|z|}\right).
\end{equation}

In order to find the asymptotics for the zeros of $\widehat{T}(z)$ we need to solve the equation
\[
e^{i(-rz+(\alpha+1/2)\pi)}+e^{izr}=O(1/n).
\]
The latter is equivalent to 
\[
\cos(rz-(\alpha+1/2)\pi/2)=O(1/n).
\]
Then $\zeta_n=\frac{\pi(3/2+\alpha+2n)}{2r}+O(1/n)$ ($n\rightarrow\infty$) is the asymptotic in question. Substituting $\zeta_n$ in formulas \eqref{pr214} and \eqref{pr215}, we find that
\begin{equation}\label{pr216}
\widehat{T}(\zeta_n) \sim \frac{e^{i(\alpha+1/2)\pi/2}{c_{0,1}}}{O(n^{\alpha+1/2})}O(1/n)+ O\left(\frac{1}{n}\right).
\end{equation}
\begin{equation}\label{pr217}
\widehat{T}'(\zeta_n) \sim \frac{e^{i(\alpha+1/2)\pi/2}{c_{0,1}}}{O(n^{\alpha+1/2})}O(1)+ O\left(\frac{1}{n}\right).
\end{equation}

That is, for $n\rightarrow \infty$ and taking into account formulas \eqref{pr216} and \eqref{pr217}, we see that $\widehat{T}(\zeta_n)=O(1/n)$ and $\widehat{T}'(\zeta_n)=O(1/n^{\alpha+1/2})$. Hence, we have 
\[
\sigma_{\zeta_n}(\widehat{T})=\frac{c_1}{\widehat{T}'(\zeta_n)}=O(n^{\alpha+1/2}).
\]
\begin{theorem}
If $f\in C_T^k[-R,R]$ for some $k\in \mathbb{N}$, $T$ is defined by formula~(\ref{eq42}) and $R>r$ then $f\in C_T^q(\mathbb{R})$ for
$q<k-(\alpha+3/2)$.
\end{theorem}
\begin{proof}
The proof is similar to the proof of Theorem~\ref{prop1}.
\end{proof}

\end{document}